\documentclass[a4paper,12pt]{amsart}
\usepackage{amssymb}
\usepackage{xypic}
\usepackage[all,cmtip]{xy}
\usepackage{amsmath}
\usepackage{chemarrow}
\usepackage[colorlinks,
pdfborder=001,
linkcolor=green,
anchorcolor=green,
citecolor=green
]{hyperref}
\usepackage{mathrsfs}
\usepackage{wasysym}
\usepackage{titletoc}
\usepackage{bm}
\usepackage{geometry}
\usepackage{colordvi}
\usepackage{color}

\allowdisplaybreaks

\DeclareMathOperator{\add}{add}
\DeclareMathOperator{\Ann}{Ann}

\DeclareMathOperator{\Aut}{Aut}

\DeclareMathOperator{\D}{\mathcal{D}}

\DeclareMathOperator{\DPic}{DPic}

\DeclareMathOperator{\End}{End}

\DeclareMathOperator{\Hom}{Hom}

\DeclareMathOperator{\id}{id}

\DeclareMathOperator{\Ker}{Ker}

\DeclareMathOperator{\Lt}{{}^{\textbf{L}}\otimes}
\DeclareMathOperator{\m}{\mathfrak{m}}
\DeclareMathOperator{\Max}{Maxspec}

\DeclareMathOperator{\p}{\mathfrak{p}}

\DeclareMathOperator{\Pic}{Pic}

\DeclareMathOperator{\RHom}{\textbf{R}Hom}

\DeclareMathOperator{\Spec}{Spec}
\DeclareMathOperator{\Supp}{Supp}

\DeclareMathOperator{\Z}{\mathbb{Z}}

\newcommand{\To}{\longrightarrow}

\numberwithin{equation}{section}

\theoremstyle{definition}

\newtheorem{thm}{Theorem}[section]
\newtheorem{prop}[thm]{Proposition}
\newtheorem{lem}[thm]{Lemma}

\newtheorem{cor}[thm]{Corollary}
\newtheorem{defn}[thm]{Definition}
\newtheorem{rk}[thm]{Remark}

\begin{document}

\title{Derived equivalences for a class of PI algebras}

\author{Quanshui Wu}
\address{School of Mathematical Sciences, Fudan University, Shanghai 200433, China}
\email{qswu@fudan.edu.cn}

\author{Ruipeng Zhu}
\address{Department of Mathematics, Southern University of Science and Technology, Shenzhen, Guangdong 518055, China}
\email{zhurp@sustech.edu.cn}

\begin{abstract}
	A description of tilting complexes is given for a class of PI algebras whose prime spectrum is canonically homeomorphic to the prime spectrum of its center. Some Sklyanin algebras are the kind of algebras considered. As an application, it is proved that any algebra derived equivalent to such kind of algebra, is Morita equivalent to it.
\end{abstract}



\maketitle

%

\section*{Introduction}

Let $A$ be a ring, $P$ be a progenerator (that is, a finitely generated projective generator) of Mod-$A$, and $\{ e_1, \cdots, e_s \}$ be a set of central complete orthogonal idempotents of $A$. For any integers $s > 0$ and $n_1 > \cdots > n_s$, it is clear that $T:=\bigoplus\limits_{i=1}^{s}Pe_i[-n_i]$ is a tilting complex over $A$ and $\End_{\D^\mathrm{b}(A)}(T)$ is Morita equivalent to $A$. If $A$ is commutative, then any tilting complex over $A$ has the above form \cite[Theorem 2.6]{Ye99} or \cite[Theorem 2.11]{RouZi03}.

Recall that a ring $A$ is called an {\it Azumaya algebra}, if $A$ is separable over its center $Z(A)$, i.e., $A$ is a projective $ (A\otimes_{Z(A)}A^{op})$-module. Any two Azumaya algebras which are derived equivalent are Morita equivalent \cite{A17}. If $A$ is an Azumaya algebra,
then there exists a bijection between the ideals of $A$ and those of $Z(A)$ via $I \mapsto I \cap Z(A)$ \cite[Corollary II 3.7]{MI}. In this case, the prime spectrum of $A$ is canonically homeomorphic to the prime spectrum of $Z(A)$.

In this note, we prove a more general result.

\begin{thm}\label{tilting complex for finite over center case}
	Let $A$ be a ring, $R$ be a central subalgebra of $A$.  Suppose that $A$ is finitely presented as $R$-module, and $\Spec(A) \to \Spec(R), \; \mathfrak{P} \mapsto \mathfrak{P} \cap R$ is a homeomorphism. 
	Then  for any tilting complex $T$  over $A$, there exists a progenerator $P$ of $A$ and a set of complete orthogonal idempotents $e_1, \cdots, e_s$ in $R$ such that in the derived category $\mathcal{D}(A)$
	$$T \cong \bigoplus\limits_{i=1}^{s}Pe_i[-n_i]$$
for some integers $s>0$ and $n_1 > n_2 > \cdots > n_s$.
\end{thm}

\begin{cor}\label{derived equiv is Morita equiv for finite over center case}
	Let $B := \End_{\D^\mathrm{b}(A)}(T)$. Then $B \cong \End_{A}(P)$, i.e., $B$ is Morita equivalent to $A$.
\end{cor}

Theorem \ref{tilting complex for finite over center case} and Corollary \ref{derived equiv is Morita equiv for finite over center case} apply to some class of Sklyanin algebras, see Corollary \ref{derived-equiv-Skl-alg}. The derived Picard group of $A$ is described in Proposition \ref{D-Pic-group}, which generalizes a result of Negron \cite{Negron17} about Azumaya algebras.

\section{Preliminaries}

Let $A$ be a ring, $\D^\mathrm{b}(A)$ be the bounded derived category of the (right) $A$-module category.
Let $T$ be a complex of $A$-modules, $\add(T)$ be the full subcategory of $\D^\mathrm{b}(A)$ consisting of objects that are direct summands of finite direct sums of copies of $T$, and  $\End_{\D^\mathrm{b}(A)}(T)$ be the endomorphism ring of $T$. A complex is called {\it perfect}, if it is quasi-isomorphic to a bounded complex of finitely generated projective $A$-modules.  $\mathcal{K}^{\mathrm{b}}(\text{proj-}A)$ is the full subcategory of $\D^{\mathrm{b}}(A)$ consisting of perfect complexes.

We first recall the definition of tilting complexes  \cite{Rickard89}, which generalizes the notion of progenerators. For the theory of tilting complexes we refer the reader to \cite{Ye20}.

In the following, $A$ and $B$ are associative ring.
\begin{defn}
	A complex $T \in \mathcal{K}^{\mathrm{b}}(\text{proj-}A)$ is called a {\it tilting complex} over $A$ if
	\begin{enumerate}
		\item $\add(T)$ generates $\mathcal{K}^{\mathrm{b}}(\text{proj-}A)$ as a triangulated category, and
		\item $\Hom_{\D^\mathrm{b}(A)}(T, T[n]) = 0$ for each $n \neq 0$.
	\end{enumerate}
\end{defn}

The following Morita theorem for derived categories, is due to Rickard.

\begin{thm}\cite[Theorem 6.4]{Rickard89}\label{derived-equivalence}
	The following are equivalent.
	\begin{enumerate}
		\item $\mathcal{D}^{\mathrm{b}}(A)$ and $\mathcal{D}^{\mathrm{b}}(B)$ are equivalent as triangulated categories.
		\item There is a tilting complex $T$ over $A$ such that $\End_{\D^\mathrm{b}(A)}(T) \cong B$.
	\end{enumerate}
\end{thm}

If $A$ and $B$ satisfy the equivalent conditions in Theorem \ref{derived-equivalence}, then $A$ is said to be {\it derived equivalent} to $B$.

The following theorem is also due to Rickard \cite[Theorem 1.6]{Ye99}, where $A$ and $B$ are two flat $k$-algebras over a commutative ring $k$.

\begin{thm}(Rickard)\label{two-side-tilting-complex}
	Let $T$ be a complex in $\D^\mathrm{b}(A \otimes_k B^{\mathrm{op}})$. The following are equivalent:
	
	(1) There exists a complex $T^{\vee} \in \D^\mathrm{b}(B \otimes_k A^{\mathrm{op}})$ and isomorphisms
		$$T \Lt_{A} T^{\vee} \cong B \text{ in } \D^\mathrm{b}(B^e)\, \text{ and }\, T^{\vee} \Lt_{B} T \cong A \text{ in } \D^\mathrm{b}(A^e)$$
where $B^e =B \otimes_k B^{\mathrm{op}}$ and $A^e =A \otimes_k A^{\mathrm{op}}$.

	(2) $T$ is a tilting complex over $A$, and the canonical morphism $B \to \Hom_{\D^\mathrm{b}(A)}(T,T)$ is an isomorphism in $\D^\mathrm{b}(B^e)$.
	
	In this case, $T$ is called a {\it two-sided tilting complex} over $A$-$B$ relative to $k$.
\end{thm}

We record some results about tilting complexes here for the convenience.

\begin{lem}\cite[Theorem 2.1]{Rickard91}\label{flat tensor products} 
	Let $A$ be an $R$-algebra and $S$ be a flat $R$-algebra where $R$ is a commutative ring. If $T$ is a tilting complex over $A$, then $T \otimes_R S$ is a tilting complex over $A \otimes_R S$.
\end{lem}
The following lemma follows directly from Lemma \ref{flat tensor products}.
\begin{lem}\cite[Proposition 2.6]{RouZi03}\label{decomposition of tilting complex by central idempotent}
	Let $T$ be a tilting complex over $A$. If $0 \neq e \in A$ is a central idempotent, then $Te$ is a tilting complex over $Ae$.
\end{lem}

The following lemma is easy to verify.

\begin{lem}\label{til-comp-lem}
	Let $T$ be a tilting complex over $A$. The following conditions are equivalent.
	
	(1) $T \cong \bigoplus\limits_{i=1}^{s}P_i[-n_i]$ in $\D^\mathrm{b}(A)$, for some $A$-module $P_i$ such that $\bigoplus\limits_{i=1}^{s} P_i$ is a progenerator of $A$.
	
	(2) $T$ is homotopy equivalent to $\bigoplus\limits_{i=1}^{s}P_i[-n_i]$ for some $A$-module $P_i$.
\end{lem}

By using a similar proof of \cite[Theorem 2.3]{Ye99} for  two-sided tilting complexes (see also \cite{Ye10} or \cite{Ye20}), the following result for tilting complexes holds.

\begin{lem}\label{equivalent condition for tilting complex being a progenerator}
	Let $T$ be a tilting complex over $A$, and $T^{\vee}:= \Hom_A(T, A)$. Let $n = \max\{ i \mid \mathrm{H}^i(T) \neq 0 \}$ and $ m = \max\{ i \mid \mathrm{H}^i(T^{\vee}) \neq 0 \}$. If $\mathrm{H}^n(T) \otimes_A \mathrm{H}^m(T^{\vee}) \neq 0$, then
	$m=-n$ and $T \cong P[m]$ in $\D^\mathrm{b}(A)$ for some progenerator $P$ of $A$.
\end{lem}
\begin{proof}
	Without loss of generality, we assume that
	$$T:= \cdots \longrightarrow 0 \longrightarrow T^{-m} \longrightarrow \cdots \longrightarrow T^n \longrightarrow 0 \longrightarrow \cdots.$$
	Since $T$ is a perfect complex of $A$-modules, $$\RHom_A(T,T) \cong T \Lt_{A} T^{\vee}.$$
	Consider the bounded spectral sequence of the double complex $T \otimes_A T^{\vee}$ (see \cite[Lemma 14.5.1]{Ye20}). Then
	$$\mathrm{H}^n(T) \otimes_A \mathrm{H}^{m}(T^{\vee}) \cong \mathrm{H}^{n+m}(T \Lt_{A} T^{\vee}) \cong \mathrm{H}^{n+m}(\RHom_A(T,T)).$$
	Since $T$ is a tilting complex, $\Hom_{\D^\mathrm{b}(A)}(T,T[i]) = 0$ for all $i \neq 0$. It follows that $n+m = 0$.
	So $T \cong T^n[-n]$.
	By Lemma \ref{til-comp-lem}, the conclusion holds.
\end{proof}

A ring $A$ is called {\it local} if $A/J_A$ is a simple Artinian ring, where $J_A$ is the Jacobson radical of $A$. If both $M_A$ and ${}_AN$ are nonzero module over a local ring $A$, then  $M \otimes_A N \neq 0$ by \cite[Lemma 14.5.6]{Ye20}.

\begin{prop}\cite[Theorem 2.11]{RouZi03}\label{derived equiv is Morita equiv for local rings}
	Let $A$ be a local ring,  $T$ be a tilting complex over $A$. Then, $T \cong P[-n]$ in $\D^\mathrm{b}(A)$, for some progenerator $P$ of $A$ and $n \in \mathbb{Z}$.
\end{prop}
\begin{proof}
It follows from Lemma \ref{equivalent condition for tilting complex being a progenerator}.
\end{proof}

Let $\Spec A$ (resp., $\Max A$) denote the prime (resp., maximal) spectrum of a ring $A$.
Let $R$ be a central subalgebra of $A$. Since the center of any prime ring is a domain, the quotient ring $R/(\mathfrak{P} \cap R)$ is a domain for any $\mathfrak{P} \in \Spec(A)$. Then there is a well defined map $\pi: \Spec A \to \Spec R, \; \mathfrak{P} \mapsto \mathfrak{P} \cap R$.
The following facts about the map $\pi$ are well known, see \cite{Bla73} for instance. For the convenience of readers, we give their proofs here.

\begin{lem}\label{Phi-lem}
	Let $A$ be a ring, $R$ be a central subring of $A$. Suppose that $A$ is finitely generated as $R$-module.
	\begin{enumerate}
		\item For any primitive ideal $\mathfrak{P}$ of $A$, $\pi(\mathfrak{P})$ is a maximal ideal of $R$. In particular, $$\pi(\Max A) \subseteq \Max R.$$
		\item If $\mathfrak{P} \in \Spec(A)$ and $\pi(\mathfrak{P})$ is a maximal ideal of $R$, then $\mathfrak{P}$ is a maximal ideal of $A$.
		\item For any multiplicatively closed subset $\mathcal{S}$ of $R$, the prime ideals of $\mathcal{S}^{-1}A$ are in one-to-one correspondence ($\mathcal{S}^{-1} \mathfrak{P} \leftrightarrow \mathfrak{P}$) with the prime ideals of $A$ which do not meet $\mathcal{S}$.
		\item $\pi: \Spec A \to \Spec R$ is surjective. 
		\item The Jacobson radical $J_R$ of $R$ is equal to $J_A \cap R$.
		\item Let $\p$ be a prime ideal of $R$. Then $A_{\p}$ is a local ring if and only if there exists  only one prime ideal $\mathfrak{P}$ of $A$ such that $\pi(\mathfrak{P}) = \p$.
		
		\item If $R$ is a Jacobson ring (that is, $\forall$ $\p \in \Spec R$, $J_{R/\p} = 0$), then so is $A$.
		
		\item If $\pi$ is injective, then $\pi(\mathcal{V}(I)) = \mathcal{V}(I \cap R)$. In this case, $\pi$ is a homeomorphism.
	\end{enumerate}
\end{lem}
\begin{proof}
	(1) Without loss of generality, we  assume that $A$ is a primitive ring. Let $V$ be a faithful simple $A$-module. For any $0 \neq x \in R$, since $Vx$ is also a non-zero $A$-module, $Vx = V$. Because $V$ is a finitely generated faithful $R$-module, $x$ is invertible in $R$. It follows that $R$ is a field. 
	
	(2) Since the prime ring $A/\mathfrak{P}$ is finite-dimensional over the field $R/(\mathfrak{P}\cap R)$, the quotient ring $A/\mathfrak{P}$ is a simple ring, that is, $\mathfrak{P}$ is a maximal ideal of $A$.
	
	(3) The proof is similar to the commutative case.
	
	(4) Suppose $\p \in \Spec R$. It follows from (3) that there exists a prime ideal $\mathfrak{P}$ of $A$ such that $\mathfrak{P}A_{\p}$ is a maximal ideal of $A_{\p}$. By (1), $\mathfrak{P}A_{\p} \cap R_{\p}$ is a maximal ideal of $R_{\p}$. Hence $\p R_{\p} \subseteq \mathfrak{P}A_{\p}$. It follows that $\mathfrak{P} \cap R = \p$.
	
	(5) By (1), $J_R \subseteq J_A \cap R$. On the other hand, by (4) and (2), $J_A \cap R \subseteq J_R$.
	
(6) It follows from (3) and (4) that $\Max A_{\p} = \{ \mathfrak{P}A_{\p} \mid \mathfrak{P} \cap R = \p, \, \mathfrak{P} \in \Spec A \}$. Then the conclusion in (6) follows.
	
	(7) Without loss of generality, we may assume that $A$ is a prime ring. Set $\mathcal{S} = R\setminus \{0\}$.
	Then $\mathcal{S}^{-1}A$ is also a prime ring which is finite-dimensional over the field $\mathcal{S}^{-1}R$. Hence $\mathcal{S}^{-1}A$ is an Artinian simple ring. It follows from that $R$ is a Jacobson ring and (5) that $J_A \cap R = J_R = 0$. Hence $\mathcal{S}^{-1}J_{A} \neq \mathcal{S}^{-1}A$, and so $\mathcal{S}^{-1}J_{A} = 0$. Hence $J_A = 0$, as every element in $\mathcal{S}$ is regular in the prime ring $A$. Therefore $A$ is a Jacobson ring.
	

	(8) Suppose $I$ is an ideal of $A$. Obviously, $\pi(\mathcal{V}(I)) \subseteq \mathcal{V}(I \cap R)$. On the other hand, let $\p \in \mathcal{V}(I \cap R) := \{ \p \in \Spec R \mid I \cap R \subseteq \p \}$. It follows from the assumption that $\pi$ is injective and (6) that $A_{\p}$ is a local ring. Hence there exists a prime ideal $\mathfrak{P}$ of $A$ such that $\mathfrak{P}A_{\p}$ is the only maximal ideal of $A_{\p}$ and $\p = \mathfrak{P} \cap R$. Since $I \cap (R \setminus \p) = \emptyset$, $IA_{\p} \subseteq \mathfrak{P} A_{\p}$ and $I \subseteq \mathfrak{P}$. Then $\p = \pi(\mathfrak{P}) \in \pi(\mathcal{V}(I))$. Hence $\pi(\mathcal{V}(I)) \supseteq \mathcal{V}(I \cap R)$. So $\pi$ is a closed map. Since $\pi$ is bijective, it is a homeomorphism.
\end{proof}

By Lemma \ref{Phi-lem}, we have the following two results, which describe the condition when $\pi$ is a homeomorphism.
\begin{lem}\label{spec-homeo}
	$\pi$ is a homeomorphism if and only if $A_{\p}$ is a local ring for any $\p \in \Spec R$.
\end{lem}

\begin{prop}\label{Max-Prime}
	Suppose that $R$ is a Jacobson ring. If the restriction map $\pi|_{\Max A}$ is injective, then $\pi$ is a homeomorphism.
\end{prop}
\begin{proof}
	If $\pi$ is not injective, then there exist two different prime ideals $\mathfrak{P}$ and $\mathfrak{P}'$ of $A$ such that $\mathfrak{P} \cap R = \mathfrak{P}' \cap R$. By Lemma \ref{Phi-lem} (7),  $A$ is a Jacobson ring. Hence, there exists a maximal ideal $\mathfrak{M}$ of $A$ such that $\mathfrak{P} \subseteq \mathfrak{M}$ and $\mathfrak{P}' \nsubseteq \mathfrak{M}$, or the other way round. Without loss of generality, assume
	$\mathfrak{P} \subseteq \mathfrak{M}$ and $\mathfrak{P}' \nsubseteq \mathfrak{M}$. Then $\pi(\mathfrak{M}) = \mathfrak{M} \cap R \supseteq \mathfrak{P} \cap R = \mathfrak{P}' \cap R$.
	
	By applying Lemma \ref{Phi-lem} (2) and (4) to the ring $A/{\mathfrak{P}'}$ with the central subalgebra $R/{\mathfrak{P}' \cap R}$, it implies that there exists a maximal ideal $\mathfrak{M}'$ of $A$, such that $\mathfrak{P}' \subseteq \mathfrak{M}'$ and $\mathfrak{M}' \cap R = \mathfrak{M} \cap R$. This contradicts to the hypothesis that $\pi|_{\Max(A)}$ is injective. So $\pi$ is a homeomorphism.
\end{proof}

\section{Some Derived equivalences imply Morita equivalences}
If $A$ is a local ring such that $A/{J_A}$ is not a skew-field and $A$ is a domain, then $A$ is not semiperfect. The localizations of the Sklyanin algebra considered in Lemma \ref{S-domain} at maximal ideals of its center are such kind of examples. The following results are needed in the proof of Theorem \ref{tilting complex for finite over center case}.
\begin{lem}\label{lemma1}
	Let $A$ be a local ring. Then there exists an idempotent element $e$ in $A$, such that any finitely generated projective right $A$-module is a direct sum of finitely many copies of $eA$.
\end{lem}
\begin{proof}
	For any finitely generated projective $A$-modules $P$, $Q$, and a surjective $A$-module morphism $f: P/PJ_A \twoheadrightarrow Q/QJ_A$, there exists  a surjective $A$-module morphism $\widetilde{f}$ so that the following diagram commutative.
$$\xymatrix{
		P \ar@{-->>}[d]^{\widetilde{f}} \ar@{->>}[r] & P/PJ_A  \ar@{->>}[d]^{f}\\
		Q \ar@{->>}[r] & Q/QJ_A
	}$$

    It follows that $Q$ is a direct summand of $P$.

	There exists a finitely generated projective $A$-module $Q\neq 0$ such that the length of $Q/QJ_A$ is smallest possible.
	By the above fact and division algorithm, any finitely generated projective $A$-module is a direct sum of finite copies of $Q$. In particular, $Q$ is a direct summand of the $A$-module $A$. Hence, there exists an idempotent element $e$ in $A$ such that $Q \cong eA$ as $A$-modules.
\end{proof}

\begin{prop}\label{a open set of fp mod in spec} Let $A$ be a ring with a central subalgebra of $R$ such that $A$ is  finitely presented as $R$-module. Suppose that $A_{\p}$ is a local ring for all $\p \in \Spec R$.
	If $M$ is a finitely presented $A$-module, then
	$$U:= \{ \p \in \Spec R \mid M_{\p} \text{ is projective over } A_{\p} \}$$
	is an open subset in $\Spec R$.
\end{prop}
\begin{proof}
 Suppose $\p \in U$. Then $M_{\p}$ is a finitely generated projective $A_{\p}$-module. Since $A_{\p}$ is a local ring, by Lemma \ref{lemma1}, there exist $x \in A$ and $s \in R\setminus\p$, such that
	$$xs^{-1} \in A_{\p} \text{ is an idempotent element and } M_{\p} \cong (xs^{-1}A_{\p})^{\oplus l}$$
for some $l \in \mathbb{N}$.

	Hence there exists $t \in R \setminus \p$ such that $(x^2-xs)st = 0$, and $xs^{-1}$ is an idempotent element in $A_{st} = A[(st)^{-1}]$. So $xs^{-1}A_{st}$ is a projective $A_{st}$-module.
There is an $A_{st}$-module morphism  $g: (xs^{-1}A_{st})^{\oplus l} \to M_{st}$ such that $g_{\p}: (xs^{-1}A_{\p})^{\oplus l} \to M_{\p}$ is the prescribed isomorphism.
	
Since $M_A$ and $A_R$ are finitely presented modules, $M$ is a finitely presented $R$-module. 
	So $M_{st}$ is a finitely presented $R_{st}$-module.
	By \cite[Proposition II.5.1.2]{Bourbaki}, 
	there exists $u \in R \setminus \p$ such that $gu^{-1}: (xs^{-1}A_{stu})^{\oplus l} \to M_{stu}$ is an isomorphism.
	It follows that $M_{stu}$ is a projective $A_{stu}$-module.
	Hence $X_{stu} := \{\mathfrak{q} \in \Spec R \mid stu \notin \mathfrak{q} \}$ is contained in $U$. Obviously, $\p \in X_{stu}$. 
	So $U$ is an open subset of $\Spec R$.
\end{proof}


%

Now we are ready to prove Theorem \ref{tilting complex for finite over center case}. 
The following proof is nothing but an adaption of the arguments of Yekutieli \cite[Theorem 1.9]{Ye10} and Negron \cite[Proposition 3.3]{Negron17} to this situation.

\begin{proof}[Proof of Theorem \ref{tilting complex for finite over center case}]
	By assumption, $\Spec A \to \Spec R, \; \mathfrak{P} \mapsto \mathfrak{P} \cap R$ is a homeomorphism. It follows from Lemma \ref{spec-homeo} that $A_{\p}$ is a local ring for any prime ideal $\p$ of $R$.
	
	There exist integers $s>0$ and $n_1 > n_2 > \cdots > n_s$ such that $\mathrm{H}^{i}(T) = 0$ for all $i \neq n_1, \cdots, n_s$, as the tilting complex $T$ is bounded.
	Obviously, $\mathrm{H}^{n_1}(T)$ is a finitely generated $R$-module. In fact, it is a finitely presented $R$-module by assumption. Hence the support set $$\Supp(\mathrm{H}^{n_1}(T)) = \mathcal{V}(\Ann_R(\mathrm{H}^{n_1}(T)))$$
	is closed in $\Spec R$.
	By Lemma \ref{flat tensor products}, $T_{\p} := T \otimes_R R_{\p}$ is a tilting complex over $A_{\p}$. If $\mathrm{H}^{n_1}(T)_{\p} \cong \mathrm{H}^{n_1}(T_{\p})$ is non-zero, then $\mathrm{H}^{n_1}(T)_{\p}$ is $A_{\p}$-projective by Proposition \ref{derived equiv is Morita equiv for local rings}. Hence
	$$\Supp(\mathrm{H}^{n_1}(T)) = \{ \p \in \Spec R \mid \mathrm{H}^{n_1}(T)_{\p} \cong \mathrm{H}^{n_1}(T_{\p}) \text{ is a non-zero projective } A_{\p}\text{-module} \}.$$
	Since $\mathrm{H}^{n_1}(T)$ is a finitely presented $A$-module, $\mathrm{H}^{n_1}(T)$ is a projective $A$-module, and
$\Supp(\mathrm{H}^{n_1}(T))$ is an open set by Proposition \ref{a open set of fp mod in spec}.
	By \cite[page 406, Theorem 7.3]{Ja}, there exists an idempotent $e_1 \in R$ such that $\Supp(\mathrm{H}^{n_1}(T)) = X_{e_1}:= \{\mathfrak{q} \in \Spec R \mid e_1 \notin \mathfrak{q} \}$. It follows that $(\mathrm{H}^{n_1}(T)/\mathrm{H}^{n_1}(T)e_1)_{\p} = 0$ for all $\p \in \Spec R$. Hence
	$\mathrm{H}^{n_1}(T)(1-e_1) = 0$. Then
	$$\mathrm{H}^{j}(T(1-e_1)) = \begin{cases}
	0, & j \neq n_2,\cdots,n_s \\
	\mathrm{H}^{n_i}(T)(1-e_1), & j = n_2,\cdots,n_s
	\end{cases}.$$
	By Lemma \ref{decomposition of tilting complex by central idempotent}, $T(1-e_1)$ is a tilting complex over $A(1-e_1)$ and $Te_1$ is a tilting complex over $Ae_1$. It follows from Proposition \ref{derived equiv is Morita equiv for local rings} that $Te_1$ is homotopy equivalent to $\mathrm{H}^{n_1}(T)[-n_1]$.

By induction on $s$, there is a complete set of orthogonal idempotents $e_1, \cdots, e_s$ in $R$ such that $Te_i$ is homotopy equivalent to $\mathrm{H}^{n_i}(T)[-n_i]$ for each $i$. Then $T = \bigoplus\limits_{i=1}^{s} Te_i = \bigoplus\limits_{i=1}^{s} \mathrm{H}^{n_i}(T)[-n_i]$. It follows from Lemma \ref{til-comp-lem} that $P:= \bigoplus\limits_{i=1}^{s} \mathrm{H}^{n_i}(T)$ is a progenerator of $A$ and $T \cong \bigoplus\limits_{i=1}^{s}Pe_i[-n_i]$ in $\D^\mathrm{b}(A)$.
\end{proof}
\begin{cor}
	If $\Spec R$ is connected, then $T \cong P[-n]$ in $\D^\mathrm{b}(A)$.
\end{cor}

\begin{cor}\cite{Negron17}\label{tilting complex over Azumaya algebra}
	Let $A$ be an Azumaya algebra. Then any tilting complex over $A$ has the form $P[n]$, where $P$ is a progenerator of $A$. If there exists a ring $B$ which is derived equivalent to $A$, then $B$ is Morita equivalent to $A$. In particular, $B$ is also an Azumaya algebra. 
\end{cor}

In the following we provide some non-Azumaya algebras which satisfy the conditions in Theorem \ref{tilting complex for finite over center case}.

\begin{defn}\cite{ATV90}
	Let $k$ be an algebraically closed field of characteristic $0$. The {\it three-dimensional Sklyanin algebras} $S = S(a, b, c)$ are $k$-algebras generated by three noncommutating variables $x,y,z$ of degree $1$, subject to relations
	$$axy+byx+cz^2 = ayz+bzy+cx^2 = azx+bxz+cy^2,$$
	for $[a:b:c] \in \mathbb{P}^2$ such that $(3abc)^3 \neq (a^3+b^3+c^3)^3$.

	The point scheme of the three-dimensional Sklyanin algebras $S$ is given by the elliptic curve
	$$E:= \mathcal{V}( abc(X^3+Y^3+Z^3) - (a^3+b^3+c^3)XYZ ) \subset \mathbb{P}^2.$$
	Let us choose the point $[1:-1:0]$ on $E$ as origin, and the automorphism $\sigma$ denotes the translation by the point $[a:b:c]$ in the group law on the elliptic curve $E$ with
	$$\sigma[x:y:z] = [acy^2-b^2xz:bcx^2-a^2yz:abz^2-c^2xy].$$
\end{defn}

\begin{lem} \label{S-domain}\cite{ATV90, ATV91}
	(1) $S$ is a Noetherian domain.

(2) $S$ is a finite module over its center if and only if the automorphism $\sigma$ has finite order.
\end{lem}
    Recently, Walton, Wang and Yakimov (\cite{WWY19}) endowed the three-dimensional Sklyanin algebra $S$, which is a finite module over its center $Z$, with a Poisson $Z$-order structure (in the sense of Brown-Gordon \cite{BrownGordon03}). By using the Poisson geometry of $S$, they analyzed all the irreducible representations of $S$. In particular, they proved the following result.

\begin{thm}\cite[Theorem 1.3.(4)]{WWY19}\label{spec-Skl-alg}
	If the order of $\sigma$ is finite and coprime with $3$, then $S/{\m S}$ is a local ring for any $\m \in \Max Z$.
\end{thm}

Here is a corollary following Theorem \ref{spec-Skl-alg} and Theorem  \ref{tilting complex for finite over center case}.
\begin{cor}\label{derived-equiv-Skl-alg}
	If the order of $\sigma$ is finite and coprime with $3$, then every tilting complex over $S$ has the form $P[n]$, where $P$ is a progenerator of $S$ and $n \in \mathbb{Z}$. Furthermore, any ring which is derived equivalent to $S$, is  Morita equivalent to $S$.
\end{cor}
\begin{proof}
	It follows from the Artin-Tate lemma that $Z$ is a finitely generated commutative $k$-algebra and so it is a Jacobson ring. Obviously, $S$ is a finitely presented $Z$-module. By Theorem \ref{spec-Skl-alg}, the restriction map $\pi|_{\Max(S)}$ is injective. It follows from Proposition \ref{Max-Prime} that $\pi: \Spec S \to \Spec Z$ is a homeomorphism. Then, the conclusions follow from Theorem \ref{tilting complex for finite over center case} and Corollary \ref{derived equiv is Morita equiv for finite over center case}.
\end{proof}

\section{Derived Picard groups}

Let $A$ be a projective $k$-algebra over a commutative ring $k$. The derived Picard group $\DPic(A)$ of an algebra $A$ was introduced by Yekutieli \cite{Ye99} and Rouquier-Zimmermann \cite{RouZi03} independently. In fact, $A$ can be assume to be flat over $k$, see the paragraph after Definition 1.1 in \cite{Ye10}.

\begin{defn}\label{defn of DPic}
	The {\it derived Picard group} of $A$ relative to $k$ is
	$$\DPic_k(A):= \frac{\{ \text{two-sided tilting complexes over } A \text{-} A \text{ relative to } k \}}{\text{ isomorphisms }},$$
	where the isomorphism is in $\D^\mathrm{b}(A\otimes_kA^{\mathrm{op}})$. The class of a tilting complex $T$ in $\DPic_k(A)$ is denoted by $[T]$. The group multiplication is induced by $- \Lt_A -$, and $[A]$ is the unit element.
\end{defn}

Let $T$ be a two-sided tilting complex over $A$-$A$ relative to $k$.  For any $z \in Z(A)$, there is an endomorphism of $T$ induced by the multiplication by $z$ on each component of $T$ (see \cite[Proposition 9.2]{Rickard89} or \cite[Propsition 6.3.2]{KonZimm98}). This defines a $k$-algebra automorphism of $Z(A)$, which is denoted by $f_T$.
The assignment $\Phi: \DPic_k(A) \to \Aut_k(Z(A)), [T] \mapsto f_T$ is a group morphism, see the paragraph in front of Definition 7 in \cite{Zim96} or \cite[Lemma 5.1]{Negron17}.


Let us first recall some definitions in \cite[Section 3.1]{Negron17}. Suppose $n \in \Gamma(\Spec Z(A), \underline{\mathbb{Z}})$, which consists of continuous functions from $\Spec Z(A)$ to the discrete space $\Z$. For any $i \in \mathbb{Z}$, $n^{-1}(i)$ is both an open and closed subset of $\Spec Z(A)$. Since $\Spec Z(A)$ is quasi-compact, there exists a set of complete orthogonal idempotents $e_{n_1}, \cdots, e_{n_s}$ of $Z(A)$ such that
$$n^{-1}(i) = \begin{cases}
X_{e_i}, & i = n_j, \, j = 1, \cdots, s \\
\emptyset, & i \neq n_1, \cdots, n_s.
\end{cases}$$
Set $e_i = 0$ for any $i \neq n_1, \cdots, n_s$, and $X_{e_i} = \emptyset$.

For any complex $T$ of $Z(A)$-modules, the shift $\Sigma^n T$ is defined by
$$\bigoplus_{i \in \Z, \, n^{-1}(i) = X_{e_i}}  Te_{i} [-i].$$

Let $\Pic_{Z(A)}(A)$ be the Picard group of $A$ over $Z(A)$. Clearly $\Gamma(\Spec Z(A), \underline{\mathbb{Z}}) \times \Pic_{Z(A)}(A)$ can be viewed as a subgroup of $\DPic_{Z(A)}(A)$ via $(n, [P]) \mapsto [\Sigma^n P]$.
The following result is proved in \cite{Negron17} under the assumption that $A$ is an Azumaya algebra.

\begin{prop}\label{D-Pic-group}
	Let $A$ be an $k$-algebra which is a finitely generated projective module over its center $Z(A)$. Suppose that 
	$\Spec A$ is canonically homeomorphic to $\Spec Z(A)$.
	Then

(1) there is an exact sequence of groups
	\begin{equation}\label{ex-seq-DPic}
	1 \To \DPic_{Z(A)}(A) \To \DPic_{k}(A) \stackrel{\Phi}{\To} \Aut_k(Z(A)).
	\end{equation}

(2) $\DPic_{Z(A)}(A) = \Gamma(\Spec Z(A), \underline{\mathbb{Z}}) \times \Pic_{Z(A)}(A)$.
\end{prop}
\begin{proof}
	It is obvious that $\Ker \Phi = \DPic_{Z(A)}(A)$. Hence \eqref{ex-seq-DPic} is a exact sequence of groups. Next we  prove $\DPic_{Z(A)}(A) = \Gamma(\Spec Z(A), \underline{\mathbb{Z}}) \times \Pic_{Z(A)}(A)$.
	
	Let $T$ be a two-sided tilting complex over $A$ relative to $k$. By Theorem \ref{tilting complex for finite over center case}, there exists a global section $n \in \Gamma(\Spec Z(A), \underline{\mathbb{Z}})$ and an $A$-progenerator $P$, such that $T \cong \Sigma^nP$ in $\D^\mathrm{b}(A)$.
	It follows from the fact that
	$$A \cong \End_{\D^\mathrm{b}(A)}(T) \cong \End_{\D^\mathrm{b}(A)}(\Sigma^nP) = \End_{\D^\mathrm{b}(A)}(P) = \End_A(P)$$
	that $P$ is an invertible $A$-$A$-bimodule with central $Z(A)$-action. By \cite[Proposition 2.3]{RouZi03}, there exists an automorphism $\sigma \in \Aut_k(A)$ such that $T \cong {^{\sigma} (\Sigma^nP)}$ in $\D^\mathrm{b}(A^e)$. If $\Phi([T]) = \id_{Z(A)}$, then $\sigma \in \Aut_{Z(A)}(A)$. Hence $T \cong \Sigma^n({^{\sigma}P})$ in $\D^\mathrm{b}(A^e)$. It follows that $\DPic_{Z(A)}(A) = \Gamma(\Spec Z(A), \underline{\mathbb{Z}}) \times \Pic_{Z(A)}(A)$.
\end{proof}

Given any algebra automorphism $\sigma$ of $Z(A)$, there is a $k$-algebra $B$ with $Z(B)=Z(A)$ and a $k$-algebra isomorphism $\widetilde{\sigma}: A \to B$, such that $\sigma=\widetilde{\sigma}|_{Z(A)}$. This determines uniquely an isomorphism class of $B$ as $Z(A)$-algebra 
 (see \cite[page 9]{F} for the definition).
Then it induces an $\Aut_k(Z(A))$-action on the $Z(A)$-algebras which are isomorphic to $A$ as $k$-algebras. The image of $\Phi$ is just the stabilizer $\Aut_k(Z(A))_{[A]}$ of the derived equivalent class (relative to $Z(A)$) of $A$.

\begin{rk}
	Suppose that $A$ is an Azumaya algebra.
	
	(1) For any $k$-algebra $B$, $B$ is derived equivalent to $A$ if and only if $B$ is Morita equivalent to $A$. So $\Aut_k(Z(A))_{[A]}$ is also the stabilizer of the Brauer class of $A$ just as in \cite{Negron17}.
	
	(2) Notice that $\DPic_{Z(A)}(A) \cong \DPic_{Z(A)}(Z(A)) \cong \Gamma(\Spec Z(A), \underline{\mathbb{Z}}) \times \Pic_{Z(A)}(A)$ is an abelian group. Hence $\DPic_{k}(A)$ is a group extension of $\Aut_k(Z(A))_{[A]}$ by $\DPic_{Z(A)}(A)$ \cite[Theorem 1.1]{Negron17}. If $A$ is not an Azumaya algebra, $\DPic_{Z(A)}(A)$ may not be abelian.
\end{rk}

\section*{Acknowledgements} The authors are very grateful to the referee for the valuable comments and suggestions. This research is partially supported by the Natural Science Foundation of China (Grant No. 11771085) and the National Key Research and Development Program of China (Grant No. 2020YFA0713200).
The authors are also very grateful to Qixiao Ma for useful discussions.


\thebibliography{plain}

\bibitem{A17} B. Antieau, Twisted derived equivalence for affine schemes, Brauer groups and obstruction
problems, Progr. Math., vol. 320, Birkh$\ddot{a}$user, Springer, Cham, 2017, pp. 7--12.

\bibitem{ATV90} M. Artin, J. Tate and M. Van den Bergh, Some algebras associated to automorphisms of elliptic curves, The Grothendieck Festschrift, vol. I, Progress in Mathematics 86 (eds P. Cartier, L. Illusie, N. M. Katz, G. Laumon, Y. I. Manin and K. A. Ribet; Birkh\"{a}user, Boston, MA, 1990) 33--85.

\bibitem{ATV91} M. Artin, J. Tate, and M. Van den Bergh, Modules over regular algebras of dimension 3, Invent. Math. 106 (1991), no. 2, 335--388.


\bibitem{Bla73} W. D. Blair, Right Noetherian rings integral over their centers, J. Algebra 27 (1973), 187--198.

\bibitem{Bourbaki} N. Bourbaki, Commutative Algebra, Elements of Mathematics, Springer, 1989, English translation ed., Chapters 1--7.

\bibitem{BrownGordon03} K. Brown, I. Gordon, Poisson orders, symplectic reflection algebras and representation theory, J. Reine Angew. Math. 559 (2003), 193--216.


\bibitem{F} A. Fr\"{o}hlich, The Picard group of noncommutative rings, in particular of orders. Trans. Amer. Math. Soc. 180 (1973), 1--45.

\bibitem{Ja} N. Jacobson, Basic Algebra II, W. H. Freeman and Company, San Francisco, 1989.


\bibitem{KonZimm98} S. K\"{o}nig, A. Zimmermann, Derived Equivalences for Group Rings, Lect. Notes Math., vol. 1685, 1998.

%
%


\bibitem{MI} F. D. Meyer and E. Ingraham, Separable Algebras over Commutative Rings, 1971 edition, Springer, 1971.


\bibitem{Negron17} C. Negron, The derived Picard group of an affine Azumaya algebra. Sel. Math. (N.S.) 23(2) (2017), 1449--1468.

\bibitem{Rickard89} J. Rickard, Morita theory for derived categories, J. Lond. Math. Soc. 39 (1989), 436--456.

\bibitem{Rickard91} J. Rickard, Derived equivalences as derived functors, J. Lond. Math. Soc. 43 (1991), 37--48.

\bibitem{RouZi03} R. Rouquier, A. Zimmermann, Picard groups for derived module categories, Proc. Lond. Math. Soc.  87 (2003), 197--225.



\bibitem{WWY19} C. Walton, X.-T. Wang, M. Yakimov, Poisson geometry of PI three-dimensional Sklyanin algebras, Proc. Lond. Math. Soc. (3) 118 (2019), no. 6, 1471--1500.

\bibitem{Ye99} A. Yekutieli, Dualizing complexes, Morita equivalence and the derived Picard group of a ring, J. London Math. Soc. (2) 60 (3) (1999), 723--746.

\bibitem{Ye10} A. Yekutieli, Derived equivalences between associative deformations, J. Pure Appl. Algebra 214 (2010), 1469--1476.

\bibitem{Ye20} A. Yekutieli, Derived categories, Cambridge Studies in Advanced Mathematics, 183. Cambridge University Press, Cambridge, 2020. xi+607 pp.

\bibitem{Zim96} A. Zimmermann, Derived Equivalences of Orders, in Proceedings of the ICRA VII, Mexico, eds: Bautista, Martinez, de la Pena. Can. Math. Soc. Conference Proceedings 18, (1996), 721--749.
\end{document}